\documentclass[12pt]{article}
\usepackage{amsmath,amssymb,amsthm,comment,array}
\usepackage{pgf,tikz}
\usetikzlibrary{arrows}
\usepackage{mathrsfs,dsfont}
\usepackage{color}
\usepackage{enumitem}
\usepackage[all]{xy}
\usepackage{graphicx}

\newtheorem{nummer}{ }[]
\newtheorem{conres}{ }[]

\newtheorem{thm}[nummer]{\sc Theorem}
\newtheorem{prp}[nummer]{\sc Proposition}
\newtheorem{lem}[nummer]{\sc Lemma}
\newtheorem{cor}[nummer]{\sc Corollary}
\newtheorem{fct}[nummer]{\sc Fact}
\newtheorem{conc}[conres]{\sc Consistency Result}

\newcounter{faelle} % (1),(2),... 
\renewcommand{\thefaelle}{\rm(\arabic{faelle})}

\newcommand{\THM}{\sc Theorem}
\newcommand{\LEM}{\sc Lemma}
\newcommand{\FCT}{\sc Fact}
\newcommand{\PRP}{\sc Proposition}

\newcommand\cV{\mathcal{V}}
\newcommand\cM{\mathcal{M}}
\newcommand\modcV{\boldsymbol{\cV}}

\newcommand\modcM{\boldsymbol{\cM}}

\renewcommand\qed{\relax\ifmmode~\hfill$\dashv$\else\unskip\nobreak~\hfill$\dashv$\fi}

\def\epsilon{\varepsilon}

\newcommand{\ie} {{\sl i.e.}}
\newcommand{\eg} {{\sl e.g.}}

\newcommand{\On}{\Omega}

\newcommand{\pair}[2]{\langle #1,#2\rangle}
\newcommand{\la}{\langle}
\newcommand{\ra}{\rangle}

\newcommand{\nG}{G}
\newcommand{\Aut}{\operatorname{Aut}}
\newcommand{\Fix}{\operatorname{Fix}}

\newcommand{\fc}{\mathfrak{c}}

\newcommand\res[1]{|_{#1}}
\newcommand\cy[2]{\operatorname{cyc}\left(#1,#2\right)}

\newcommand{\fm}{\mathfrak{m}}
\newcommand{\fn}{\mathfrak{n}}
\newcommand{\fone}{\mathfrak{1}}

\newcommand{\ffive}{\mathfrak{5}}

\newcommand\ran{\operatorname{ran}}
\newcommand\dom{\operatorname{dom}}
\newcommand\fin{\operatorname{fin}}
\newcommand\seq{\operatorname{seq}}
\newcommand\Sym{\operatorname{sym}}
\newcommand\iseq{\operatorname{seq}^{\text{\rm\tiny{1-1}}}}

\newcommand{\card}[1]{|#1|}

\newcommand\ZFC{\text{\sf ZFC}}
\newcommand\ZF{\text{\sf ZF}}
\newcommand\ZFA{\text{\sf ZFA}}
\newcommand\AC{\text{\sf AC}}

\newcommand\pwr{\mathscr{P}}
\newcommand\nF{\mathscr{F}}
\newcommand\nE{\mathscr{E}}
\newcommand\nS{\mathscr{S}}

\newcommand{\blcb} {\big{\{}}
\newcommand{\brcb} {\big{\}}}

\newcommand\subs{\subseteq}

\renewcommand{\phi}{\varphi}
\renewcommand{\theta}{\vartheta}

\newcommand{\NDiag}{\text{\bfseries\sffamily{N}}}
\newcommand{\NrevD}{\text{\bfseries\sffamily{\reflectbox{N}}}}
\newcommand{\CDiag}{\text{\bfseries\sffamily{C}}}
\newcommand{\CrevD}{\text{\bfseries\sffamily{\reflectbox{C}}}}
\newcommand{\ZDiag}{\text{\bfseries\sffamily{Z}}}
\newcommand{\ZrevD}{\text{\bfseries\sffamily{\reflectbox{Z}}}}

\begin{document}

\begin{center}
{\Large\sc Four Cardinals and Their Relations in\/ {\sf ZF}}\\[1.8ex]
{\small Lorenz Halbeisen,  
Riccardo Plati, Salome Schumacher, and Saharon Shelah\footnote{Research 
partially supported by 
ISF 1838/19: The Israel Science Foundation (ISF) (2019-2023), 
and Rutgers 2018 DMS 1833363: NSF DMS Rutgers visitor program (PI S. Thomas) (2018-2021)\\
%
%NSF grant No.\;DMS 1833363}$^{,}$
%\footnote{Research partially supported 
%by Israel Science Foundation (ISF) grant No.\;1838/19 (2019-2023).\\ 
\indent Paper 1220 on Shelah's publication list.}}
\end{center}

%%%%%%%%%%%%%%%%%%%%%%%%%%%%%%%%%%%%%%%%%%%%%%%%%%%%%%%%%%%%%%%%%%%%%%%%%%

% submitted to arxiv: 23 September, submit/3943006

%==========================
\begin{quote}
{\small {\bf Abstract.} 
For a set $M$, $\fin(M)$ denotes the set of all finite subsets of $M$,
$M^2$ denotes the Cartesian product $M\times M$, 
$[M]^2$ denotes the set of all $2$-element subsets of $M$,
and $\iseq(M)$ denotes the set of all finite sequences
without repetition which can be formed with elements of $M$. 
Furthermore, for a set
$S$, let $\card{S}$ denote the cardinality of $S$. 
Under the assumption that the four cardinalities
$\card{[M]^2}$, $\card{M^2}$, $\card{\fin(M)}$, $\card{\iseq(M)}$
are pairwise distinct and pairwise comparable in~$\ZF$, 
there are six possible linear orderings between these four
cardinalities. We show that at least five of the
six possible linear orderings are consistent with~$\ZF$.
}
\end{quote}

\begin{quote}
\small{{\bf key-words\/}: permutation models, cardinals in $\ZF$, consistency results}
% 03E35 Consistency and independence results
% 03E10 Ordinal and cardinal numbers
% 03E25 Axiom of choice and related propositions
\end{quote}

\section{Introduction}

Let $M$ be a set. Then $\fin(M)$ denotes the set of all finite subsets of $M$,
$M^2$ denotes the Cartesian product $M\times M$, 
$[M]^2$ denotes the set of all $2$-element subsets of $M$,
$\iseq(M)$ denotes the set of all finite sequences without repetitions
which can be formed with elements of $M$, and $\seq(M)$
denotes the set of all finite sequences which can be formed with elements of $M$
(where repetitions are allowed).

Furthermore, for a set $A$, let $\card{A}$ denote the cardinality of $A$. 
We write $\card A=\card B$, if there exists a bijection between $A$ and $B$,
and we write $\card A\le\card B$, if there exists a bijection between $A$
and a subset $B'\subs B$ (\ie, $\card A\le\card B$ if and only if there 
exists an injection from $A$ into $B$). Finally, we write $\card A<\card B$
if $\card A\le\card B$ and $\card A\neq\card B$. By the 
{\sc Cantor-Bernstein Theorem},
which is provable in $\ZF$ only (\ie, without using the Axiom of Choice),
we get that $\card A\le\card B$ and $\card A\ge\card B$ implies 
$\card A=\card B$. 

Let $\fm:=|M|$, and let $[\fm]^2:=\card{[M]^2}$, $\fm^2:=\card{M^2}$, 
$\fin(\fm):=\card{\fin(M)}$, $\iseq(\fm):=\card{\iseq(M)}$,
and $\seq(\fm):=\card{\seq(M)}$. 
Concerning these cardinalities, in $\ZF$ we obviously have $\iseq(\fm)\le 
\seq(\fm)$, $[\fm]^2\le \fin(\fm)$ and $\fm^2\le\iseq(\fm)$, 
where the latter relations are visualized
by the following diagram (in the diagram, $\fn_1$ is below 
$\fn_2$ if $\fn_1\le\fn_2$):

{\small
\[
\xymatrix@R=10mm@C=3mm{
\fin(\fm)\ar@{-}[d]& &\iseq(\fm)\ar@{-}[d]\\ 
[\fm]^2& &\fm^2
}
\]

Moreover, for \emph{finite\/} cardinals $\fm$ with $\fm\ge\mathfrak{5}$ we have
$$[\fm]^2<\fm^2<\fin(\fm)<\iseq(\fm)\,,$$
and in the presence of the Axiom of Choice ({\ie}, in $\ZFC$),
for every infinite cardinal $\fm$ we have
$$[\fm]^2=\fm^2=\fin(\fm)=\iseq(\fm)\,.$$
It is natural to ask whether some of these equalities can be proved also
in $\ZF$, {\ie}, without the aid of $\AC$. Surprisingly, this is not the
case. In~\cite{carla}, a permutation model was constructed in which for
an infinite cardinal $\fm$ we have $\seq(\fm)<\fin(\fm)$
(see~\cite[Thm.\,2]{carla} or~\cite[Prp.\,7.17]{cst1st}). 
As a consequence we obtain that the existence
of an infinite cardinal $\fm$ such that $\iseq(\fm)<\fin(\fm)$ is 
consistent with $\ZF$. This consistency result was modified
to the existence of an infinite cardinal 
$\fm$ for which $\fm^2<[\fm]^2$ (see~\cite[Prp.\,7.18]{cst1st}), 
and later, it was strengthened to the existence of an infinite cardinal 
$\fm$ for which $\seq(\fm)<[\fm]^2$ (see~\cite{weird} or~\cite[Prp.\;8.28]{cst}).
The consistency of $\fin(\fm)<\iseq(\fm)$ for infinite cardinals $\fm$ can be obtained
with the {\it Ordered Mostowski Model\/} (see, for example, 
\cite[Related\;Result\,48,\;p.\,217]{cst}), in which there is an 
infinite cardinal $\fm$ with $$[\fm]^2<\fm^2<\fin(\fm)<\iseq(\fm)\,.$$
Consistency results as well as $\ZF$-results concerning the relations between these 
cardinals with other cardinals can be found, for example, 
in~\cite{Shen1,Shen2} or~\cite{berlin}.

Concerning the four cardinalities $[\fm]^2$, $\fm^2$, $\fin(\fm)$, and $\iseq(\fm)$, 
a question which arises naturally is whether for some infinite cardinal $\fm$, 
$\fin(\fm)<\fm^2$ is consistent with $\ZF$ (see \cite[Related\;Result\,20,\;p.\,133]{cst}).
Moreover, assuming that $\fm$ is infinite and the four cardinalities 
$[\fm]^2$, $\fm^2$, $\fin(\fm)$, $\iseq(\fm)$ are pairwise distinct and pairwise 
comparable in~$\ZF$, one may ask which linear orderings on these four cardinalities
are consistent with~$\ZF$.

Since for all cardinals $\fm$,
we cannot have $[\fm]^2>\fin(\fm)$ or $\fm^2>\iseq(\fm)$, there are only the following
six linear orderings on these four cardinalities which might be consistent with $\ZF$
(where for two cardinals $\fn_1$ and $\fn_2$, 
$\fn_1\longrightarrow\fn_2$ means $\fn_1<\fn_2$). 

%%% 2x2 - Diagramme

{\ }\hfill\parbox{.3\textwidth}{
\[
%% A
\xymatrix@R=10mm@C=3mm{
\fin(\fm)\ar@{<-}[d]\ar@{->}[drr]& &\iseq(\fm)\ar@{<-}[d]\\ 
[\fm]^2& &\fm^2
}
\]
\centerline{\small{\it Diagram\;}$\NDiag$}
}
\hfill
%%%%%%%%%%%%%%%%%%%%%%%%%%%%%
\parbox{.3\textwidth}{
\[
%% C
\xymatrix@R=10mm@C=3mm{
\fin(\fm)\ar@{<-}[d]\ar@{->}[rr]& &\iseq(\fm)\\ 
[\fm]^2& &\fm^2\ar@{->}[ll]
}
\]
\centerline{\small{\it Diagram\;}$\CDiag$}
}
\hfill
%%%%%%%%%%%%%%%%%%%%%%%%%%%%%
\parbox{.3\textwidth}{
\[
%% E
\xymatrix@R=10mm@C=3mm{
\fin(\fm)\ar@{<-}[rr]& &\iseq(\fm)\\ 
[\fm]^2\ar@{->}[urr]& &\fm^2\ar@{->}[ll]
}
\]
\centerline{\small{\it Diagram\;}$\ZDiag$}
}
\hfill{\ }

%%%%%%%%%%%%%%%%%%%%%%%%%%%%%%%%%%%%%%%%

{\ }\hfill\parbox{.3\textwidth}{
\[
%% F
\xymatrix@R=10mm@C=3mm{
\fin(\fm)\ar@{<-}[d]& &\iseq(\fm)\ar@{<-}[d]\ar@{->}[dll]\\ 
[\fm]^2& &\fm^2
}
\]
\centerline{\small{\it Diagram\;}$\NrevD$}
}
\hfill
%%%%%%%%%%%%%%%%%%%%%%%%%%%%%
{\ }\hfill\parbox{.3\textwidth}{
\[
%% D
\xymatrix@R=10mm@C=3mm{
\fin(\fm)\ar@{<-}[rr]& &\iseq(\fm)\ar@{<-}[d]\\ 
[\fm]^2& &\fm^2\ar@{<-}[ll]
}
\]
\centerline{\small{\it Diagram\;}$\CrevD$}
}
\hfill
%%%%%%%%%%%%%%%%%%%%%%%%%%%%%
{\ }\hfill\parbox{.3\textwidth}{
\[
%% B
\xymatrix@R=10mm@C=3mm{
\fin(\fm)\ar@{->}[rr]& &\iseq(\fm)\\ 
[\fm]^2& &\fm^2\ar@{<-}[ll]\ar@{->}[ull]
}
\]
\centerline{\small{\it Diagram\;}$\ZrevD$}
}
\hfill{\ }

Below we show that each of the five diagrams $\NDiag$, $\ZDiag$, $\NrevD$, $\CrevD$, 
$\ZrevD$ is consistent with $\ZF$. 
% and we give a few facts about Diagram\;$\CDiag$.
      % fertig
\section{Permutation Models}

In order to show, for example, that for some infinite cardinals $\fm$ and $\fn$, 
\hbox{$\fm<\fn$} is consistent with $\ZF$, 
by the {\sc Jech-Sochor Embedding Theorem} 
(see, for example, \cite[Thm.\,6.1]{Jechchoice} or~\cite[Thm.\,17.2]{cst}), it is
enough to construct a permutation model in which this statement holds.
The underlying idea of permutation models, which will be models of 
set theory with atoms ($\ZFA$),
is the fact that a model $\modcV\models\ZFA$ does not distinguish between the
atoms, where atoms are objects which do not have
any elements but which are distinct from the empty set. 
The theory $\ZFA$ is essentially the same as that of $\ZF$
(except for the definition of ordinals, where we have to require that
an ordinal does not have atoms among its elements). Let $A$ be a set.
Then by transfinite recursion on the ordinals $\alpha\in\On$ we can define
the $\alpha$-power ${\pwr}^\alpha(A)$ of~$A$ and
${\pwr}^\infty(A)=\bigcup_{\alpha\in\On}{\pwr}^\alpha(A)$. Like
for the cumulative hierarchy of sets in $\ZF$, one can show that if
$\modcM$ is a model of $\ZFA$ and $A$ is the set of atoms of
$\modcM$, then $\modcM={\pwr}^\infty(A)$. The
class $M_0:={\pwr}^\infty(\emptyset)$ is a model of $\ZF$
and is called the {\bf kernel}. Notice that all ordinals belong to the kernel.
By construction we obtain that every permutation of the set of atoms induces an
automorphism of $\modcM$, where the sets in the kernel are fixed.

Permutation models were first
introduced by Adolf~Fraenkel and, in a precise version (with
supports), by Andrzej~Mostowski. The version with filters, which we will follow
below, is due to Ernst~Specker (a detailed introduction to permutation models 
can be found, for example, in~\cite[Ch.\,8]{cst} or~\cite{Jechchoice}).

In order to construct a permutation model, we usually start with 
a set of atoms $A$ and then define a group $\nG$ of permutations or 
automorphisms of $A$. 

The permutation models we construct below are of the following simple
type: For each finite set $E\in\fin(A)$, let $$\Fix_{\nG}(E):=\blcb\pi\in\nG:
\forall a\in E\,(\pi a=a)\brcb\,,$$ and let $\nF$ be the 
filter of subgroups of $\nG$ generated by the subgroups 
$\{\Fix_{\nG}(E): E\in\fin(A)\}$. In other words, $\nF$ is the set
of all subgroups $H\le\nG$, such that there exists a finite set
$E\in\fin(A)$, such that $\Fix_G(E)\le H$.

For a set $x$, let 
$$\Sym_{\nG}(x):=\{\pi \in\nG :\pi x\,=\,x\}$$ 
where $$\pi x=\begin{cases}
\ \emptyset &\text{if $x=\emptyset$,}\\[1ex]
\ \pi a&\text{if $x=a$ for some $a\in A$,}\\[1ex]
\blcb\pi y:y\in x\brcb&\text{otherwise.}
\end{cases}$$
Then, a set $x$ is {\it symmetric\/} if and only if
there exists a set of atoms $E_{x}\in\fin(A)$, such that
$$\Fix_{\nG}(E_x)\le\Sym_{\nG}(x).$$ We say that $E_x$ is a {\it
support\/} of $x$. Finally, let $\modcV$ be the class 
of all hereditarily symmetric objects; then
$\modcV$ is a transitive model of $\ZFA$. 
We call $\modcV$ a {\it permutation model}.
So, a set $x$ belongs to the permutation model
$\modcV$ (with respect to $\nG$ and $\nF$), 
if and only if $x\subseteq\modcV$ and $x$ has a finite support $E_x\in\fin(A)$.
Because every $a\in A$ is symmetric, we get that each atom $a\in A$
belongs to $\modcV$.
 % fertig
\subsection{A Model for Diagram \protect\NDiag}% N

We first show that in every model for Diagram\;$\NDiag$, we have that 
the cardinality $\fm$ is transfinite.

\begin{lem}\label{lem:N}
If\/ $\fin(\fm)\le\fm^2$ for some\/ $\fm\ge\ffive$, then\/ $\aleph_0\le\fm$.
\end{lem}

\begin{proof} Let $A$ be a set of cardinality $\fm\ge\ffive$ and assume that
$h:\fin(A)\to A^2$ is an injection. First we choose a
$5$-sequence $S_5:=\la a_1,\ldots,a_5\ra$ of pairwise distinct elements of~$A$. 
The ordering of $S_5$ induces an ordering on $P_5:=\fin(\{a_1,\ldots,a_5\})$, and since
$|h[P_5]|=2^5$ and $2^5>5^2$, there exists a first set $u\in P_5$ such that for 
$\pair xy=h(u)$, the set $D_6:=\{x,y\}\setminus \{a_1,\ldots,a_5\}$ is non-empty. 
If $x\in D_6$, let $a_6:=x$, otherwise, let $a_6:=y$. Now, let $S_6:=\la a_1,\ldots,a_6\ra$
and $P_6:=\fin(\{a_1,\ldots,a_6\})$. As above, we find a $u\in P_6$ such that for 
$\pair xy=h(u)$, the set $D_7:=\{x,y\}\setminus \{a_1,\ldots,a_6\}$ is non-empty.
If $x\in D_7$, let $a_7:=x$, otherwise, let $a_7:=y$. Proceeding this way, we 
finally have an injection from $\omega$ into $A$, which shows that $\aleph_0\le\fm$.
\end{proof}

\begin{prp}\label{prp:fin-to-one}
If\/ $\aleph_0\le\fm$ for some cardinal\/ $\fm=\card A$, then there exists a finite-to-one function\/ $g:\seq(A)
\to\fin(A)$.
\end{prp}

\begin{proof}
By the assumption, there exists an injection $h:\omega\to A$,
and for each $i\in\omega$, let $x_i:=h(i)$, let $B=\{x_{i}:i\in\omega\}$, and let
$C:=\{x_{2i}:i\in\omega\}$. Notice that 
$$\iota(a)=
\begin{cases}
a &\text{if $a\in A\setminus B$,}\\
x_{2i+1} &\text{if $a=x_i$,}
\end{cases}$$
is a bijection between $A$ and $A\setminus C$. Thus,
it is enough to construct a finite-to-one function $g:\seq(A\setminus C)
\to\fin(A)$. Let $s=\la a_0,\ldots,a_{n-1}\ra\in \seq(A\setminus C)$ and let
$\ran(s):=\{a_0,\ldots,a_{n-1}\}$. The sequence $s$ gives us in a natural way
an enumeration of $\ran(s)$, and with respect to this enumeration we can encode the 
sequence $s$ by a natural number $i_s\in\omega$. Now, let
$g(s):=\ran(s)\cup\{x_{2i_s}\}$. Then, since there are just finitely many 
enumerations of $\ran(s)$, $g$ is a finite-to-one function.
\end{proof}

The following result is just a consequence of {\PRP}\;\ref{prp:fin-to-one} 
and {\LEM}\;\ref{lem:N}.

\begin{cor}
If\/ $\fin(\fm)\le\fm^2$ for some\/ $\fm=\card A\ge\ffive$, then there 
exists a finite-to-one function\/ $g:\seq(A)\to\fin(A)$.
\end{cor}

We now introduce the technique we intend to use in order to build a permutation 
model from which it will follow that for some infinite cardinal $\fm$,
the relation $\fin(\fm)<\fm^2$ is consistent with $\ZF$. Notice that this 
relation is the main feature of Diagram~$\NDiag$ and that this relation implies
that $\aleph_0\le\fm$.
In the next section, we shall use a similar permutation model in order to show 
the consistency of Diagram~$\ZDiag$ with $\ZF$. 
%To increase readability, we define the setting specifically for one diagram at a time, and afterwards 
%we briefly comment on how this construction can be carried on in a more general way.

Let $K$ be the class of all the pairs $(A,h)$ such that $A$ is a (possibly empty) set
and $h$ is an injection $h\colon\fin(A)\setminus\{\emptyset\}\to A^2$. We will also refer to the elements 
of $K$ as models. We define a partial ordering $\leq$ on $K$ by stipulating
$$(A,h)\leq (B,f)\iff
A\subs B\;\wedge\;h\subs f\;\wedge\; %\ran(h)=\ran(f)\cap A^2\;\wedge\;
\ran\left(f|_{\fin(B)\setminus\fin(A)}\right)\subs B^2\setminus A^2\,.$$ 
When the functions involved are clear from the context, with a slight abuse 
of notation we will just write $A\leq B$ instead of $(A,h)\leq(B,f)$ 
and $A\in K$ instead of $(A,h)\in K$.

Before proceeding, we give two preliminary definitions. Given a model $(M,f)$ 
and a countable subset $A\subs M$, 
we define the \textit{closure\/} $\textrm{cl}(A,M)$ as 
the smallest superset of $A$ that is closed under $f$ and pre-images 
with respect to the same function. Constructively, we can characterize 
$\textrm{cl}(A,M)$ as a countable union as follows: 
Define $\textrm{cl}_0=\textrm{cl}_0(A,M):=A$ and, for all $i\in\omega$,
    $$\textrm{cl}_{i+1}= \textrm{ cl}_i \;\cup 
    \underset{\parbox{9ex}{\scriptsize $p\in\textrm{fin(cl}_i)$\\
    $p\neq\emptyset$}}{\hspace*{-3ex}\bigcup}\hspace*{-3ex}\ran(f(p))\; 
    \cup \bigcup_{q\in(\textrm{cl}_i)^2\cap\ran(f)} f^{-1}(q)$$
in order to finally define 
$\textrm{cl}(A,M):=\bigcup_{i\in\omega}\textrm{cl}_i$. 
Furthermore, we set a standardized way 
to extend a \textit{partial} model $(A,f')$, where $f'$ is only a 
partial function, to an element of $K$: Consider $(A,f')$, 
where $A$ is a set and $f'$ is an injection with $\dom(f')\subseteq \textrm{fin}(A)\setminus\{\emptyset\}$ and $\ran(f')\subseteq A^2$. Let $(M_0,f'_0)=(A,f')$ and, for $j\in\omega$, define inductively   
$(M_{j+1},f'_{j+1})$ as follows: $M_{j+1}$ is the fully 
disjoint union $$M_j\quad\sqcup
\underset{\parbox{17ex}{\scriptsize $
P\in \textrm{fin}(M_j)\setminus\dom(f'_j)$\\
$P\neq\emptyset$}}
{\hspace*{-8ex}\bigsqcup}
\hspace{-7ex}\{a_P,b_P\}.$$ 
For what concerns the injection $f'_{j+1}$, we naturally 
require the inclusion $f'_j\subseteq f'_{j+1}$, 
as well as the equality $\dom(f'_{j+1})=\textrm{fin}(M_j)\setminus\{\emptyset\}$, where for
$P\in \textrm{fin}(M_j)\setminus\dom (f'_j)$ with $P\neq\emptyset$, we define 
$f'_{j+1}(P):=(a_P,b_P)$.
We are now in the position of defining the 
{\it plain extension\/} of $(A,f')$ as 
$$(M,f):=\left(\bigcup_{j\in\omega}M_j,\;\bigcup_{j\in\omega}f'_j\right).$$

Given the previous definitions, we remark that given a model $M\in K$ and a countable subset $A\subseteq M$, we have that $\textrm{cl}(A,M)\leq M$, which proves the following:

\begin{fct}\label{fct:existence}
For every countable subset $A$ of a model $M\in K$, there is a countable model $N$ such that $A\subseteq N\leq M$.
\end{fct}

\begin{prp}[CH]\label{prp:limit}
There is a model\/ $M_*$ of cardinality\/ $\fc$ in $K$ such that:
\begin{itemize}
    \item $M_*$ is\/ $\aleph_1$-universal, i.e., if\/ $N\in K$ is countable then\/ $N$ 
    is isomorphic to some\/ $N_*\leq M_*$.
    \item $M_*$ is\/ $\aleph_1$-homogeneous, i.e., if\/ $N_1,N_2\leq M_*$ are 
    countable and\/ $\pi\colon N_1\to N_2$ is an isomorphism then there exists 
    an automorphism\/ $\pi_*$ of $M_*$ such that\/ $\pi\subs\pi_*$.
    \item If\/ $N\leq M_*$ and\/ $A\subs M_*$ are countable, then there 
    is an automorphism\/ $\pi$ of\/ $M_*$ that fixes\/ $N$ pointwise, 
    such that\/ $\pi(A)\setminus N$ is disjoint from $A$. 
\end{itemize}
\end{prp}

\begin{proof}
We construct the model $M_*$ by induction on~$\omega_1$, where we assume that $\omega_1=\fc$.
Let $M_0=\emptyset$. When $M_\alpha$ is already defined for some $\alpha\in\omega_1$, we can define 
$$C_\alpha:=\blcb N\leq M_\alpha:N\in K\;\textrm{and $N$ is countable}\brcb.$$

The construction of $M_{\alpha+1}$, starting from $M_\alpha$, consists of a disjoint 
union of two differently built sets of models. First, for each element $N\in C_\alpha$, 
let $S_N$ be a system of representatives for the $\textit{strong}$ isomorphism classes 
of all the models $M\in K$ such that $N\leq M$ with $M$ countable.
Here, by $\textit{strong}$ we mean that, 
for two models $M_1$ and $M_2$ with $N\le M_1,M_2$,
it is not enough to be isomorphic in order to belong to 
the same class, but we require that there exists an isomorphism between $M_1$ and $M_2$ 
that fixes $N$ pointwise, which we can express by saying that \textit{$M_1$ is isomorphic to $M_2$ 
over $N$}. We first extend $M_\alpha$ by the set
$$M_\alpha'=\bigsqcup_{N\in C_\alpha}\bigsqcup_{M\in S_N}M\setminus N,$$
where {``}\,$\bigsqcup$\,{''} indicates that we have a \textit{disjoint union}, and now we define $M_{\alpha+1}$ as the plain extension of $M_\alpha\sqcup M_\alpha'$. Finally, for non-empty limit ordinals $\delta$ define $M_\delta=\cup_{\alpha\in\delta}M_\alpha$, and let $$M_*=\bigcup_{\alpha\in\omega_1}M_\alpha.$$
It remains to show that the model $M_*$ has the required properties:
First we notice that $M_*$ has cardinality $|M_*|=\mathfrak{c}$, as required,
and since, by construction, $M_1$ is $\aleph_1$-universal,
$M_*$ is also $\aleph_1$-universal. In order to show that $M_*$ 
is $\aleph_1$-homogeneous, we make use of a back-and-forth argument. Let $N_1,N_2\leq M_*$ be countable models and $\pi\colon N_1\to N_2$ an isomorphism. Let $\{x_\alpha:\alpha\in\omega_1\}$ be 
an enumeration of the elements of $M_*$ and let $I_0:=N_1$. If $x_{\delta_1}$ 
is the first element (with respect to 
this enumeration) in $M_*\setminus I_0$, 
then, by {\FCT}\;\ref{fct:existence}, there exists a countable model 
$I'_1\leq M_*$ such that $I_0\leq I'_1$ and $x_{\delta_1}\in I'_1$. Similarly, there is a countable model $J'_1$ with $N_2\leq J'_1\leq M_*$ such that 
there exists an isomorphism $\pi'_1\colon I'_1\to J'_1$ with $\pi\subseteq\pi'_1$. Now, let $x_{\gamma_1}$ be the first element in $M_*\setminus J'_1$: for the same reason as above we can find countable models $J_1, I_1$ such that $I'_1\leq I_1\leq M_*$ and $J'_1\leq J_1\leq M_*$, together with $x_{\gamma_1}\in J_1$ and the fact that there exists an isomorphism $\pi_1\colon I_1\to J_1$ with $\pi'_1\subseteq\pi_1$.
Proceed inductively with $x_{\delta_{\alpha+1}}$ being the first element in 
$M_*\setminus I_\alpha$ and find countable models $I_\alpha\leq I'_{\alpha+1}\leq M_*$, 
$J_\alpha\leq J'_{\alpha+1}\leq M_*$ and an isomorphism 
$\pi'_{\alpha+1}\colon I'_{\alpha+1}\to J'_{\alpha+1}$ with $x_{\delta_{\alpha+1}}\in I'_{\alpha+1}$ and $\pi_\alpha\subseteq\pi'_{\alpha+1}$. As in the second part of the base step, let $x_{\gamma_{\alpha+1}}$ be the first element in 
$M_*\setminus J'_{\alpha+1}$ and find countable models $I'_{\alpha+1}\leq I_{\alpha+1}\leq M_*$, 
$J'_{\alpha+1}\leq J_{\alpha+1}\leq M_*$, with an isomorphism 
$\pi_{\alpha+1}\colon I_{\alpha+1}\to J_{\alpha+1}$ such that $x_{\gamma_{\alpha+1}}\in J_{\alpha+1}$ and $\pi'_{\alpha+1}\subseteq\pi_{\alpha+1}$. We naturally take the union at limit stages and finally obtain
$\pi_*=\cup_{\alpha\in\omega_1}\pi_\alpha$, which is the required automorphism of $M_*$.

To show the last property of the theorem, let $N\leq M_*$ and $A\subs M_*$ 
be both countable. Since the cofinality of $\omega_1$ is greater than $\omega$, 
we can find by construction both a countable model $M$ satisfying the properties 
$A\subseteq M$, $N\leq M\leq M_*$ and a further countable model $M'$ with $N\leq M'\leq M_*$ 
such that $M'\cap(A\setminus N)=\emptyset$, and such that there exists an 
isomorphism $i\colon M\to M'$ with $i$ fixing $N$ pointwise. Now, by $M_*$ 
being $\aleph_1$-homogeneous we obtain an automorphism $i_*$ extending $i$, 
as required.
\end{proof}

As anticipated, the construction of the previous theorem does not exploit any 
particular property of the functions $h\colon\fin(A)\setminus\{\emptyset\}\to A^2$. In fact, 
the construction is an analogue of a Fra\"{\i}ss\'e limit as it relies 
on similar properties, like, for example, a modified version of 
the {\sl Disjoint Amalgamation Property\/} (DAP) of~$K$, where we require that embeddings 
between structures $f\colon(A,h)\to(B,g)$ are allowed only when, according to our 
previous definition, $A\leq B$. Indeed, exactly the same construction can be carried 
out in the alternative framework of models $(A,f,g,h)$, where $A$ is a set and we have 
three injections $f\colon A^2\to[A]^2$, $g\colon[A]^2\to\iseq(A)$ and 
$h\colon\iseq(A)\to\fin(A)$, which will be used below 
to show the consistency of Diagram~$\ZDiag$ with~$\ZF$.

Given {\PRP}\;\ref{prp:limit}, we consider the permutation 
model $\modcV_\NDiag$ that arises 
naturally by considering the elements of the $\aleph_1$-universal 
and $\aleph_1$-homogeneous model $M_*$ 
as the set of atoms and its automorphisms $\Aut(M_*)$ as the 
group $G$ of permutations. In particular,
each permutation in $G$ preserves the injection 
$h\colon\fin(M_*)\setminus\{\emptyset\}\to M_*^2$ 
that the model $(M_*,h)$ comes with.

We are now ready to prove the following result.

\begin{thm}\label{thm:model_N}
Let\/ $M_*$ be the set of
atoms of\/ $\modcV_\NDiag$ and let\/ $\fm=|M_*|$. Then
$$\modcV_\NDiag\models [\fm]^2<\fin(\fm)<\fm^2<\iseq(\fm)\,.$$ 
\end{thm}

\begin{proof}
The existence of an injection $i\colon\fin(M_*)\to M_*^2$ in $\modcV_\NDiag$ follows from the existence of the injection $h\colon\fin(M_*)\setminus\{\emptyset\}\to M_*^2$ given by the specific permutation model, together with the fact that $h$ cannot be surjective, which will be shown later. So, we only need to 
prove that in $\modcV_\NDiag$, there is no reverse injection from $M_*^2$ into $\fin (M_*)$, 
and that there are no injections from $\fin(M_*)$ into $[M_*]^2$ 
or from $\iseq(M_*)$ into $M_*^2$.

In order to show that there is neither an injection from $M_*^2$ into 
$\fin (M_*)$, nor an injection from $\iseq(M_*)$ into $M_*^2$,
assume towards a contradiction that $\modcV_\NDiag$ contains an injection
$f_1\colon M_*^2\to\fin (M_*)$ or an injection $f_2\colon\iseq(M_*)\to M_*^2$. Let $S$ 
be a finite support of both functions~$f_1$ and $f_2$ (if they exist). In other words, 
$S\in\fin (M_*)$ and for each automorphism $\pi\in\Fix_{G}(S)$ 
we have $\pi(f_1)=f_1$ and  $\pi(f_2)=f_2$, respectively. 
Let $N_1$ be a countable model in $K$ with $S\subs N_1\leq M_*$. 
Let $(N_2,g)$ be a countable model in $K$ 
such that $(N_1,h|_{N_1})\le (N_2,g)$,
constructed as follows: The domain of $N_2$ is 
the disjoint union $$N_2=N_1\sqcup\{x,y,z\}\sqcup\{a_i:i\in\omega\}\,.$$
Furthermore, we define the injection $g\colon\fin(N_2)\setminus\{\emptyset\}\to N_2^2$ such that $g\supseteq h|_{N_1}$
and for $E\in\fin (N_2)\setminus\fin (N_1)$ we define $g(E)=\la e_1,e_2\ra$ such
that $g$ is injective and satisfies the following conditions (recall that since $N_2$ is 
countable, also $\fin(N_2)$ is countable):
\begin{itemize}
    \item If $E\cap\{x,y,z\}=\emptyset$ then $\la e_1,e_2\ra=\la a_n,a_m\ra$ for some 
    $n,m\in\omega$.
    % recall that there are only countably many such finite subsets $E$.
    \item If $|E\cap\{x,y,z\}|=1$, then $\la e_1,e_2\ra =\la u,a_k\ra$ 
    for some $k\in\omega$, where $u$ is the unique element in $E\cap\{x,y,z\}$. 
    %Again, notice that it is possible for $g$ to be an injection since $N_2$ is countable;
    \item If $|E\cap\{x,y,z\}|=2$, then $\la e_1,e_2\ra=\la v,a_k\ra$ for some $k\in\omega$, 
    where $v$ is the unique element in $\{x,y,z\}\setminus(E\cap\{x,y,z\})$.
    \item If $|E\cap\{x,y,z\}|=3$ then $\la e_1,e_2\ra=\la a_n,a_m\ra$ for some $n,m\in\omega$. 
\end{itemize}
Notice that there are automorphisms of $(N_2,g)$ that just permute ${x,y,z}$ and
fix all other elements of $N_2$ pointwise.
By construction of $M_*$, we find a model $N_2'\in K$ such that 
$N_1\leq N_2'\leq M_*$ and $N_2'$ is isomorphic to $N_2$ over~$N_1$.
For this reason we can refer to $N_2$ as a legit submodel of $M_*$ that extends $N_1$ 
in the way we described. Let us now consider $f_1(\la x,y\ra)$, where we assumed in $\modcV_\NDiag$
the existence of an injection $f_1\colon M_*^2\to\fin (M_*)$ with finite support~$S$. 
If $f_1(\la x,y\ra)\subseteq N_1$ or $f_1(\la x,y\ra)\nsubseteq N_2$, then we can
apply the third property of {\PRP}\;\ref{prp:limit} with respect to $f_1(\la x,y\ra)$ 
and $N_1$ and~$N_2$ respectively, which gives us a contradiction.
If $\{x,y\}\subs f_1(\la x,y\ra)$ or $\{x,y\}\cap f_1(\la x,y\ra)=\emptyset$,
we could swap $x$ and $y$ while fixing every other element of $N_2$ 
pointwise and get $f_1(\la x,y\ra)=f_1(\la y,x\ra)$, which would imply that
$f_1$ is not injective. So, assume that
$|\{x,y\}\cap f_1(\la x,y\ra)|=1$ and without loss
of generality assume that $\{x,y\}\cap f_1(\la x,y\ra)=\{x\}$. 
Now, if $z\in f_1(\la x,y\ra)$, {\ie}, $\{x,z\}\subs f_1(\la x,y\ra)$, 
we similarly obtain a contradiction by swapping $z$ and $x$, 
while if $z\notin f_1(\la x,y\ra)$ we get a contradiction by swapping $z$ and $y$. 
This shows that $f_1$ cannot belong to $\modcV_\NDiag$.

For what concerns $f_2$, let us consider the set $\nS$ consisting of sequences
without repetition of $\{x,y,z\}$ of length~$2$ or~$3$.
Notice that $|\nS|=12$. Now, for each element $s\in\nS$, if $f_2(s)=\la a,b\ra$, then $a$ and
$b$ are such that $a\neq b$ and $\la a,b\ra\in\{x,y,z\}^2$\,---\,notice that otherwise, for example, 
if $\{a,b\}\cap\{x,y\}=\emptyset$, then we can swap $x$ and $y$ and hence move $s$ 
without moving $\la a,b\ra$, which is not consistent with $S$ being a support of $f_2$. 
We get the conclusion by noticing that, because of this restriction, 
there are only six possible images of elements of $\nS$, which implies 
that $f_2$ cannot be an injection.

It remains to show that in $\modcV_\NDiag$ there are no injections 
from $\fin(M_*)$ into $[M_*]^2$. For this, assume towards a contradiction 
that there exists such a function $f_3$
in $\modcV_\NDiag$ and assume that $S$ is a finite support of~$f_3$. Then, 
let $N_1$ be a countable model in $K$ with 
$S\subs N_1\leq M_*$. We will construct a countable model $(N_2,g)\in K$ 
satisfying $(N_1,h|_{N_1})\leq (N_2,g)\leq M_*$ with a finite subset 
$u\in\fin(N_2\setminus N_1)$ such that, for all $\la x,y\ra\in N^2_2\setminus N^2_1$, 
one of the following holds:
\begin{itemize}
\item there is no finite set $E\in\fin(N_2)\setminus\{\emptyset\}$ with $h(E)=\la x,y\ra$;
\item there exists an automorphism $\pi$ of $N_2$ over $N_1$ with $\pi(u)\neq u$ 
and $\pi\{x,y\}=\{x,y\}$.
\end{itemize}
Let $u=\{a_0,b_0,c_0\}$ be disjoint from $N_1$ and define $G^1_0=N_1\sqcup\{a_0,b_0,c_0\}$. 
Now, for each finite set $E\in\fin(G^1_0)\setminus\{\emptyset\}$ which is not in the domain of 
$h_0=h|_{N_1}$, that is, for each finite set $E\in\fin(G^1_0)$ with 
$E\cap\{a_0,b_0,c_0\}\neq\emptyset$, let $\{x_E,y_E\}$ be a pair of new elements  
and define $$G^*_0\;:=\;G^1_0\;\sqcup\hspace{-2ex}\bigsqcup_{E\in\fin(G^1_0)\setminus\dom(h_0),\,E\neq\emptyset}
\hspace{-2ex}\{x_E,y_E\}\qquad\text{and}\qquad
h^1_0:=h_0\;\cup\hspace{-2ex}\bigcup_{E\in\fin(G^1_0)\setminus
\dom(h_0),\,E\neq\emptyset}\hspace{-2ex}\blcb\big{\la} E,\la x_E,y_E\ra\big{\ra}\brcb.$$ 
Let now $G^2_0$ be an extension of $G^*_0$ 
by adding a copy of $G^1_0\setminus N_1$, where the ``copy function'' 
is denoted by $\tau_0$. Notice that at this stage, $G^1_0\setminus N_1=\{a_0,b_0,c_0\}$.
More formally, $G^2_0=G^*_0\sqcup\{\tau_0(a):a\in 
G^1_0\setminus N_1\}$, together with an extension of $h^1_0$ defined as 
$$h^2_0:=h^1_0\;\cup\hspace{-2ex}\bigcup_{E\in\fin(G^1_0)\setminus\dom(h_0),\,E\neq\emptyset}
\hspace{-2ex}\blcb\big{\la} \tau_0(E),\la y_E,x_E\ra\big{\ra}\brcb,$$ 
where, given $E\in\fin(G^1_0)\setminus\dom(h_0)$ with $E\neq\emptyset$, $\tau_0(E)$ is defined 
as $$\tau_0(E):=(E\cap N_1)\sqcup\{\tau_0(a):a\in E\setminus N_1\}.$$ 
Notice that if $a\in G^1_0\setminus N_1$, then $\tau_0(a)\in G^2_0\setminus G^*_0$. 
The construction carried out so far is actually the first of countably many 
analogous extension steps we will consequently apply in order to consider 
the union of all the progressive extensions. 
That is, assume that for some 
$i\in\omega$ we have already defined $G^2_i$ and $h^2_i$. Define $G^1_{i+1}=G^2_i$ 
and, for each non-empty finite set $E\in\fin(G^1_{i+1})$ which is not in the domain of $h^2_i$, 
consider a pair of new elements $\{x_E,y_E\}$ and define
$$G^*_{i+1}:=G^1_{i+1}\;\sqcup\hspace{-2ex}\bigsqcup_{E\in\fin(G^1_{i+1})
\setminus\dom(h^2_i),\,E\neq\emptyset}\hspace{-2ex}\{x_E,y_E\}\qquad\text{and}$$ $$\qquad
h^1_{i+1}:=h^2_i\;\cup\hspace{-2ex}
\bigcup_{E\in\fin(G^1_{i+1})\setminus\dom(h^2_i),\,E\neq\emptyset}\hspace{-4ex}
\blcb\big{\la}E,\la x_E,y_E\ra\big{\ra}\brcb.$$
Let now $G^2_{i+1}$ be an extension of $G^*_{i+1}$ by adding a copy of 
$G^1_{i+1}\setminus N_1$, where the ``copy function'' is now $\tau_{i+1}$. 
More formally, $G^2_{i+1}:=G^*_{i+1}\;\sqcup\{\tau_{i+1}(a):a\in G^1_{i+1}\setminus N_1\}$, 
together with an extension of $h^1_{i+1}$ defined as 
$$h^2_{i+1}:=h^1_{i+1}\;\cup\hspace{-2ex}\bigcup_{E\in\fin(G^1_{i+1})
\setminus\dom(h^2_i),\,E\neq\emptyset}\hspace{-5ex}\blcb\big{\la}\tau_{i+1}(E),\la y_E,x_E\ra\big{\ra}\brcb\;\cup\hspace{-2ex}\bigcup_{E\in\dom(h^2_i)
\setminus\fin(N_1)}\hspace{-2ex}\blcb\big{\la}\tau_{i+1}(E),\tau_{i+1}(h^2_i(E))\big{\ra}\brcb,$$ 
where, again, given $E\in\fin(G^1_{i+1})\setminus\fin(N_1)$, $\tau_{i+1}(E)$ 
is defined as $$\tau_{i+1}(E):=(E\cap N_1)\sqcup\{\tau_{i+1}(a):
a\in E\setminus N_1\},$$ 
for which we newly remark that if 
$a\in G^1_{i+1}\setminus N_1$, then $\tau_{i+1}(a)\in G^2_{i+1}\setminus 
G^*_{i+1}$. 
Notice that 
every automorphism of $(G_{i+1}^1,h_i^2)$ can be extended to an automorphism of $(G_{i+1}^*,h_{i+1}^1)$\,---\,this is because
$G_{i+1}^*\setminus G_{i+1}^1$ consists of pairs $\{x,y\}$, 
each of which corresponds to a unique non-empty finite subset of $G_{i+1}^1$ and {\it moves\/} according to this finite subset. 
Finally, we conclude that every automorphism of 
$(G_{i+1}^*,h_{i+1}^1)$ can be extended to an automorphism 
of $(G_{i+1}^2,h_{i+1}^2)$, which follows from the following 
two facts: $G_{i+1}^2\setminus G_{i+1}^*$ consists of a copy through $\tau_{i+1}$ of $G_{i+1}^1\setminus N_1$, and therefore
supports the same automorphisms. Furthermore, every
automorphism of $(G_{i+1}^*,h_{i+1}^1)$ is an extension of some automorphism of $(G_{i+1}^1,h_i^2)$, which follows from the fact that for all $E\in\fin(G_{i+1}^*)\setminus\emptyset$, we have $E\in\dom(h_{i+1}^1)$ if and only if $E\cap(G_{i+1}^*\setminus G_{i+1}^1)=\emptyset$.

Now, let $$N_2:=\bigcup_{i\in\omega}
G^1_i\qquad\text{and}\qquad g:=\bigcup_{i\in\omega}h^1_i\,.$$
We claim that $(N_2,g)$ satisfies the required properties. 
Indeed, if $\la x,y\ra\in N^2_2\setminus N_1^2$ and there is some 
finite set $E\in\fin(N_2)\setminus\{\emptyset\}$ with $g(E)=h(E)=\la x,y\ra$, then by 
construction of $g$ we necessarily have $\la x,y\ra\in (N_2\setminus N_1)^2$ and the following fact: either there exists some index $n\in\omega$ such that $\la x,y\ra\in (G^*_n\setminus G^1_n)^2$, or there are indices $n,k\in\omega$ with $k>n$ such that for some ordered pair $\la x',y'\ra \in (G^*_n\setminus G^1_n)^2$ we have $\la x,y\ra = \bigl\la\tau_k(x'),\tau_k(y')\bigr\ra$. Each of the two conditions implies that there 
exists an automorphism $\pi$ of $N_2$ over $N_1$ acting as follows:
$\pi\la x,y\ra=\la y,x\ra$ and 
$\pi u=\pi\{a_0,b_0,c_0\}=\blcb\tau_n(a_0),\tau_n(b_0),\tau_n(c_0)\brcb$, 
which in particular means $\pi\{x,y\}=\{x,y\}$ and $\pi u\neq u$, 
as desired. We can finally consider the image $f_3(u)=\{x,y\}$: 
If $\{x,y\}\nsubseteq N_2$ or $\{x,y\}\subseteq N_1$, then we can
apply the third property of {\PRP}\;\ref{prp:limit} with respect to $\{x,y\}$ 
and $N_1$ and~$N_2$ respectively, which gives us a contradiction.
Thus $\{x,y\}\subseteq N_2$ and $\{x,y\}\nsubseteq N_1$, and if there exists some finite 
set $E\in\fin(N_2)\setminus\emptyset$ with $g(E)=h(E)=\la x,y\ra$, then by the reasoning 
above we find that some automorphism of $N_2$ over $N_1$ does not 
preserve $f_3$, a contradiction. In every other case, we consider 
$\textrm{cl}(N_1\cup\{x,y\},M_*)$ and notice that, since for no 
$E\in\fin(N_2)\setminus\emptyset$ we have $h(E)=\la x,y\ra$ or $h(E)=\la y,x\ra$, then we claim that $u$ 
cannot be a subset of $\textrm{cl}(N_1\cup\{x,y\},M_*)$, which allows 
us to fix $\textrm{cl}(N_1\cup\{x,y\},M_*)$ pointwise, while not 
preserving $u$, a contradiction as well.
In order to prove the claim, let $U_0:=u$ and for 
$i\in\omega$, let $$U_{i+1}:=U_i\cup\{\tau_i(a):a\in U_i\}$$ and
define $U:=\bigcup_{i\in\omega}U_i$. Furthermore, we define a rank-function
$\operatorname{rk}:N_2\to\omega\cup\{-\infty\}$ by stipulating 
$$\operatorname{rk}(a):=
\begin{cases}
-\infty & \text{if $a\in N_1$},\\
0 & \text{if $a\in U$},\\
n+1 & \text{if $a\in\{x_E,y_E\}$, where $E\in\fin(N_2)\setminus\fin(N_1)$}\\
& \text{and $\max\{\operatorname{rk}(b): b\in E\}=n$},\\
n & \text{if $a=\tau_k(b)$ for some $k\in\omega$ with $\operatorname{rk}(b)=n$.}
\end{cases}$$
Since for no $E\in\fin(N_2)\setminus\emptyset$ we have 
$h(E)=\la x,y\ra$ or $h(E)=\la y,x\ra$, for any
$a\in\textrm{cl}(N_1\cup\{x,y\},M_*)\setminus N_1$
we have $\operatorname{rk}(a)\ge\min\blcb\operatorname{rk}(x),
\operatorname{rk}(y)\brcb$. Thus, the only way that $u\subs 
\textrm{cl}(N_1\cup\{x,y\},M_*)$ would be that $\{x,y\}\cap U\neq\emptyset$.
However, even in the case when $\{x,y\}\subs U$ ({\eg}, $\{x,y\}\subs 
\{a_0,b_0,c_0\}$), by the definition of $\tau_i$, at least one of the
elements of $\{a_0,b_0,c_0\}$ does not belong to $\textrm{cl}(N_1\cup\{x,y\},M_*)$. In particular, $u\nsubseteq\textrm{cl}(N_1\cup\{x,y\},M_*)$,
which proves the claim.
\end{proof}

So, the model $\modcV_{\NDiag}$ witnesses the following

\begin{conc} 
The existence of an infinite cardinal\/ $\fm$ satisfying
%$$\fm^2<[\fm]^2<\seq(\fm)<\fin(\fm)$$
\[\xymatrix@R=10mm@C=3mm{
\fin(\fm)\ar@{->}[drr]& &\seq(\fm)\\ 
[\fm]^2\ar@{->}[u]& &\fm^2\ar@{->}[u]
}
\]
is consistent with\/ $\ZF$.
\end{conc}
\subsection{A Model for Diagram \protect\ZDiag}

We are now going to set an analogue framework to the one for Diagram~$\NDiag$, 
just with the definitions adapted, in order to show the consistency of 
Diagram~$\ZDiag$. In fact, as mentioned above, 
we can state the same proposition, guaranteeing the existence of a suitable 
$\aleph_1$-universal and $\aleph_1$-homogeneous model.

Let $K$ be the class of all the quadruples $(A,f,g,h)$ such that $A$ is a (possibly empty) 
set and $f,g,h$ are the following three injections:
$$f\colon A^2\to[A]^2\qquad 
g\colon[A]^2\to\textrm{seq}^{1-1}(A)\qquad h\colon\textrm{seq}^{1-1}(A)\to\textrm{fin}(A),$$ where the function $h$ satisfies $h(\emptyset)=\emptyset$. As before, we define a partial ordering $\leq$ on $K$ 
by stipulating $(A,f_1,g_1,h_1)\leq (B,f_2,g_2,h_2)$ if and only if 
\begin{itemize}
\item $A\subseteq B$, 
\item $f_1\subseteq f_2$,
$\ran\left(f_2|_{B^2\setminus A^2}\right)\subs [B]^2\setminus[A]^2$, 
\item $g_1\subseteq g_2$,
$\ran\left(g_2|_{[B]^2\setminus [A]^2}\right)\subs\iseq(B)\setminus\iseq(A)$,
\item $h_1\subseteq h_2$,
$\ran\left(h_2|_{\iseq(B)\setminus \iseq(A)}\right)\subs\fin(B)\setminus\fin(A)$.
\end{itemize}
%$$\ran(f_1)=\ran(f_2)\cap [A]^2,\quad\ran(g_1)=\ran(g_2)\cap 
%\iseq(A),\quad\ran(h_1)=\ran(h_2)\cap \fin(A).$$  

\begin{prp}[CH]\label{prp:AtomsZ}
There is a model\/ $M_*$ of cardinality\/ $\mathfrak{c}$ in $K$ such that:
\begin{itemize}
    \item $M_*$ is\/ $\aleph_1$-universal, i.e., if\/ $N\in K$ is countable 
    then\/ $N$ is isomorphic to some\/ $N_*\leq M_*$.
    \item $M_*$ is\/ $\aleph_1$-homogeneous, i.e., if\/ $N_1,N_2\leq M_*$ are 
    countable and\/ $\pi\colon N_1\to N_2$ is an isomorphism then there exists 
    an automorphism\/ $\pi_*$ of\/ $M_*$ such that\/ $\pi\subs\pi_*$.
    \item If\/ $N\leq M_*$ and\/ $A\subs M_*$ are countable, then there is an 
    automorphism\/ $\pi$ of\/ $M_*$ over\/ $N$ such that\/ 
    $\pi(A)\setminus N$ is disjoint from\/ $A$. 
\end{itemize}
\end{prp}
\begin{proof}
The proof is essentially the same as the one of {\PRP}\;\ref{prp:limit}.
\end{proof}

We define $\modcV_\ZDiag$ as the permutation model obtained by setting 
the elements of the $\aleph_1$-universal and $\aleph_1$-homogeneous model $M_*$ 
as the set of atoms and its automorphisms $\Aut(M_*)$ as the group $G$ of permutations. 
In particular, each permutation in $G$ preserves
the injections $f,g,h$ that the model $(M_*,f,g,h)$ comes with.

\begin{thm}
Let\/ $M_*$ be the set of
atoms of\/ $\modcV_\ZDiag$ and let\/ $\fm=|M_*|$. Then
$$\modcV_\ZDiag\models \fm^2<[\fm]^2<\iseq(\fm)<\fin(\fm)\,.$$ 
\end{thm}
\begin{proof}
The existence of the required injections is clear by the definition of the 
model. Thus, it remains to prove that there are no reverse injections. 
Here, as in {\THM}\;\ref{thm:model_N}, we
will make heavily use of the concepts of \textit{closure} and of plain extension: given a model $(M,f,g,h)$ 
and a countable subset $A\subs M$, 
we define the \textit{closure\/} $\textrm{cl}(A,M)$ as 
the smallest superset of $A$ that is closed under $f,g,h$ and pre-images 
with respect to the same functions. Constructively, we can characterize 
$\textrm{cl}(A,M)$ as a countable union as follows: 
Define $\textrm{cl}_0=\textrm{cl}_0(A,M):=A$ and, for all $i\in\omega$,
\begin{align*}
    \textrm{cl}_{i+1}= \textrm{ cl}_i &\;\cup \bigcup_{p\in(\textrm{cl}_i)^2}\hspace{-1ex}f(p)\; 
    \cup 
    \bigcup_{q\in[\textrm{cl}_i]^2}\hspace{-1ex}\ran(g(q))\;\cup
    \hspace{-2ex}\bigcup_{s\in\iseq(\textrm{cl}_i)}\hspace{-2.5ex}h(s) \\[1.0ex]
    & \cup 
    \hspace{-2ex}\bigcup_{q\in[\textrm{cl}_i]^2\cap\ran(f)} 
    \hspace{-3.5ex}\ran(f^{-1}(q)) \hspace{2ex}\cup 
    \hspace{-3ex}\bigcup_{s\in\textrm{seq}^{1-1}(\textrm{cl}_i)\cap\ran(g)}\hspace{-5ex}g^{-1}(s) \hspace{2ex}\cup 
    \hspace{-3ex}\bigcup_{r\in\textrm{fin}(\textrm{cl}_i)\cap\ran(h)}
    \hspace{-3ex}\ran(h^{-1}(r))\,,
\end{align*}
in order to finally define $\textrm{cl}(A,M):=\bigcup_{i\in\omega}
\textrm{cl}_i$. 
Furthermore, we set a standardized way 
to extend a \textit{partial} model $(A,f',g',h')$, where $f',g',h'$ are only 
partial functions, to an element of $K$: Consider $(A,f',g',h')$, 
where $A$ is a set and $f',g',h'$ are injections with $h'(\emptyset)=\emptyset$ and
\begin{center}
$$\begin{array}{ccc}
  \dom(f')\subseteq A^2,&\dom(g')\subseteq [A]^2, & \dom(h')\subseteq\textrm{seq}^{1-1}(A)\\[1.4ex]
  \ran(f')\subseteq [A]^2, &\ran(g')\subseteq\textrm{seq}^{1-1}(A),& \ran(h')\subseteq\textrm{fin}(A).
\end{array}
$$
\end{center}
Let $(M_0,f'_0,g'_0,h'_0)=(A,f',g',h')$ and, for $j\in\omega$, define inductively   
$(M_{j+1},f'_{j+1},g'_{j+1},h'_{j+1})$ as follows: $M_{j+1}$ is the fully 
disjoint union $$M_j\quad\sqcup\bigsqcup_{P\in M_j^2\setminus\dom(f'_j)}
\hspace{-3ex}\{a_P,b_P\}\quad\sqcup\bigsqcup_{Q\in[M_j]^2\setminus\dom(g'_j)}\hspace{-3.4ex}\{a_Q,b_Q,c_Q\}
\quad\sqcup\bigsqcup_{R\in\iseq(M_j)\setminus\dom(h'_j)}\hspace{-5ex}\{a_R,b_R,c_R\}.$$ 
For what concerns the injections $f'_{j+1},g'_{j+1},h'_{j+1}$, we naturally 
require the inclusions $f'_j\subseteq f'_{j+1}$, $g'_j\subseteq g'_{j+1}$, and $h'_j\subseteq h'_{j+1}$, 
as well as the equalities $\dom(f'_{j+1})=M_j^2$, $\dom(g'_{j+1})=[M_j]^2$, 
and $\dom (h'_{j+1})=\textrm{seq}^{1-1}(M_j)$, respectively, where for
$P\in M_j^2\setminus\dom (f'_j)$, 
$Q\in[M_j]^2\setminus\dom (g'_j)$, and $R\in\textrm{seq}^{1-1}(M_j)
\setminus\dom (h'_j)$, we define  $$f'_{j+1}(P):=\{a_P,b_P\}\,,\qquad g'_{j+1}(Q):=
\langle a_Q,b_Q,c_Q\rangle\,,\qquad h'_{j+1}(R):=\{a_R,b_R,c_R\}\,.$$
We are now in the position of defining the plain extension of $(A,f',g',h')$ as 
$$(M,f,g,h):=\left(\bigcup_{j\in\omega}M_j,\;\bigcup_{j\in\omega}f'_j,\;\bigcup_{j\in\omega}g'_j,\;
\bigcup_{j\in\omega}h'_j\right),$$ 
and we can finally prove, in three analogous steps, 
that neither of the three injections of the model $(M_*,f,g,h)$ 
admits a reverse injection.

Assume there is an injection $i\colon [M_*]^2\to M_*^2$ with finite support 
$S$. Let $N_1\in K$ 
be a countable model such that $N_1\leq M_*$ and $S\subs N_1$. 
Let $\{x,y\}\in [M_*]^2$ with $N_1\cap\{x,y\}=\emptyset$, 
let $M_0=N_1\sqcup\{x,y\}$, and let $N_2$ be the plain extension of $M_0$.
Without loss of generality we can assume that $N_2\leq M_*$. 
Consider $\la a,b\ra=i(\{x,y\})$. 
Then $\{a,b\}\nsubseteq N_1$ and $\{a,b\}\subs N_2$, since otherwise we could 
apply the third property of {\PRP}\;\ref{prp:AtomsZ} with respect to $\{a,b\}$ 
and $N_1$ and $N_2$, 
respectively. Moreover, $\{a,b\}\cap\{x,y\}=\emptyset$, since otherwise,
({\eg}, $a=x$), we could swap $x$ and $y$ while fixing $S$ pointwise, but this would not preserve the injection $i$, which was assumed to have support $S$, a contradiction.
Furthermore, we have 
\begin{equation}
\{x,y\}\;\subs\;\textrm{cl}(N_1\cup\{a,b\},M_*)\tag{$*$}
\end{equation}
since otherwise we could 
apply the third property of {\PRP}\;\ref{prp:AtomsZ} with respect to $\{x,y\}$ and
\hbox{$\textrm{cl}\bigl(N_1\cup\{a,b\},M_*\bigr)\leq M_*$.} 
Now, this last inclusion implies that 
$a\neq b$ and that $\{a,b\}=f(\la x',y'\ra)$ for 
some $\la x',y'\ra\in N^2_2\setminus N^2_1$. 
To see this, notice first that since $N_1\cup \{a,b\}\subs N_2$,
we build the closure 
\hbox{$\textrm{cl}\bigl(N_1\cup\{a,b\},M_*\bigr)$} within the
plain extension~$N_2$, and recall that for $\{u,v\}\in [N_2]^2\setminus [M_0]^2$
we have $g(\{u,v\})=\la x_1,x_2,x_3\ra$ where $\la x_1,x_2,x_3\ra\in 
\iseq(N_2\setminus M_0)$, and that for $\la x_1,\ldots,x_n\ra\in\iseq(N_2)\setminus
\iseq(M_0)$ we have $h(\la x_1,\ldots,x_n\ra)\in [N_2\setminus M_0]^3$.
If there are no $x',y'$ such that $f(\la x',y'\ra)=\{a,b\}$, then, since
$\{a,b\}\subs N_2$ and $\{a,b\}\nsubseteq N_1$, $\{a,b\}$  cannot be a superset of any $\ran\bigl(g(\{u,v\})\bigr)$ for some $u,v$, 
or of any $h(\la x_1,\ldots,x_n\ra)$ for some $x_1,\ldots,x_n$, so we have that $\{x,y\}\nsubseteq\textrm{cl}(N_1\cup\{a,b\},M_*)$,
which is a contradiction to~$(*)$.
Now, since $f(\la x',y'\ra)=\{a,b\}$, by the construction of the plain extension $N_2$
we find an automorphism $\pi$ of $N_2$ 
that fixes $N_1\cup\{x,y\}$ pointwise and for which we have $\pi(a)=b$ and $\pi(b)=a$.
Hence, $i(\pi\{x,y\})=\la a,b\ra\neq\la b,a\ra=\pi i(\{x,y\})$, which is a
contradiction.

Assume there is an injection $i\colon \mathrm{seq}^{1-1}(M_*)\to [M_*]^2$ with 
finite support $S$. Let $N_1\in K$ be a countable model such that
$N_1\leq M_*$ and $S\subs N_1$. Let $\la x,y,z\ra\in\iseq(M_*)$ with 
$N_1\cap\{x,y,z\}=\emptyset$, let $M_0=N_1\sqcup\{x,y,z\}$, and 
let $N_2$ be the plain extension of $M_0$. Finally, let $\{a,b\}=i(\la x,y,z\ra)$. 
Then, a contradiction follows by noticing\,---\,with similar arguments as 
above\,---\,that necessarily 
$\{x,y,z\}\not\subs\textrm{cl}\bigl(N_1\cup\{a,b\},M_*\bigr)$.
So, similarly as above, there is an automorphism $\pi$ of $M_*$ which fixes 
$\textrm{cl}(N_1\cup\{a,b\},M_*)$ pointwise, but $\pi(\{x,y,z\})\neq\{x,y,z\}$.

Finally, assume there is an injection $i\colon\fin(M_*)\to\mathrm{seq}^{1-1}(M_*)$ 
with finite support $S$. Let $N_1\in K$ 
be a countable model such that $N_1\leq M_*$ and $S\subs N_1$. 
Let $\{x,y,z\}\in [M_*]^3$ be such that \hbox{$N_1\cap\{x,y,z\}=\emptyset$,} 
let $M_0=N_1\sqcup\{x,y,z\}$, and let $N_2$ be the plain extension of $M_0$.
Consider $\la a_j:j\in n\ra=i(\{x,y,z\})$ for some $n\in\omega$. It is easy 
to see that we must have $\{a_j:j\in n\}\cap (N_2\setminus M_0)\neq\emptyset$, 
and as before it must also hold $\{x,y,z\}\subs
\textrm{cl}\bigl(N_1\cup\{a_j:j\in n\},M_*\bigr)$. 
In what follows, for a natural number $n$, we will refer to a cyclic permutation of order $n$ by using the term ``$n$-cycle''. Our next step is to prove that a $3$-cycle $\pi$ applied to $\{x,y,z\}$ 
cannot leave each element of $\{a_j:j\in n\}\cap (N_2\setminus M_0)$ unchanged, for at least one automorphism $\sigma$ of $N_2$ extending $\pi$ does not fix any element of $N_2\setminus M_0$, as the following argument shows: Assume there is such an automorphism $\sigma$ fixing a first appearing element $c\in M_n\setminus M_{n-1}$ in the construction of the plain extension of $M_0$. Notice that $\sigma$ moves every unordered pair, ordered pair and injective sequence with non-empty intersection with $\{x,y,z\}$. Now, by looking at the possible ways that $c$ could have appeared in $N_2$, one sees that if $c$ is fixed, then some pair $p\in [M_{n-1}]^2$ such that $c\in\ran(g(p))$ also satisfies $\sigma(p)=p$. Now, since $c$ was the first appearing element being fixed, we can say that for $p=\{a,b\}\in [M_{k}]^2$, for which $p=\{a,b\}\notin [M_{k-1}]^2$ holds, it is true that $a,b$ are the first appearing elements being swapped by $\sigma$. Similarly as before, this implies that there are two pairs $p',p''\in[M_{k-1}]^2$ such that $\sigma(p)=p'$ and $\sigma(p')=p$, which in turn requires either the existence of two ordered pairs in $M_{k-1}^2$ swapped by $\sigma$, either contradicting the fact that $a$ and $b$ were the first appearing swapped elements, or implying the existence of a four-element set $\{s,r,t,p\}\subseteq N_2\setminus M_0$ on which $\sigma$ acts as a four-cycle, but we can extend $\pi$ to at least one $\sigma$ such that $\sigma^3 = \textrm{Id}_{N_2}$, which does not allow four-cycles, and this concludes the proof.
\end{proof}

So, the model $\modcV_{\ZDiag}$ witnesses the following

\begin{conc} 
The existence of an infinite cardinal\/ $\fm$ satisfying
%$$\fm^2<[\fm]^2<\seq(\fm)<\fin(\fm)$$
\[\xymatrix@R=10mm@C=3mm{
\fin(\fm)\ar@{<-}[rr]& &\iseq(\fm)\\ 
[\fm]^2\ar@{->}[urr]& &\fm^2\ar@{->}[ll]
}
\]
is consistent with\/ $\ZF$.
\end{conc}

\subsection{A Model for Diagram \protect\NrevD} 

We show that Diagram~\protect{\NrevD} holds in the model constructed
in~\cite{weird} (see also~\cite[p.\,209\,ff]{cst}), 
where $\fm$ is the cardinality of the set of atoms of that model.

The atoms of the permutation model $\modcV_{\NrevD}$ for 
Diagram~\protect{\NrevD} are constructed as follows:
\begin{itemize}
\setlength{\itemsep}{.5ex}
\item[($\alpha$)] Let $A_0$ be an arbitrary infinite set.
\item[($\beta$)] $\nG_0$ is the group of {\it all\/} permutations of $A_0$.
\item[($\gamma$)] $A_{n+1}:=A_n\cup\blcb (n+1,p,\varepsilon):
p\in\bigcup_{k=0}^{n+1} A_n^k\wedge\varepsilon\in\{0,1\}\brcb$.
\item[($\delta$)] $\nG_{n+1}$ is the subgroup of the
permutation group of $A_{n+1}$ containing all permutations 
$\sigma$ for which there are
$\pi_\sigma\in\nG_n$ and $\varepsilon_{\sigma,p}\in\{0,1\}$ such that
\begin{eqnarray*}
\sigma(x)=
\begin{cases}
\pi_\sigma(x) &\mbox{if $x\in A_n$,}\cr\noalign{\vspace{3pt}}
(n+1,\pi_\sigma(p),\varepsilon_{\sigma,p}+_2\varepsilon) 
&\mbox{if $x=(n+1,p,\varepsilon)$,}
\end{cases}
\end{eqnarray*}
where for $p=\langle p_0,\ldots,p_{l-1}\rangle\in
\bigcup_{0\le k\le n+1} A_n^k$, 
$\pi_\sigma(p):=\langle \pi_\sigma(p_0),\ldots,
\pi_\sigma(p_{l-1})\rangle$
and $+_2$ denotes addition modulo $2$.
\end{itemize}

Let $A:=\bigcup\{A_n:n\in\omega\}$ be the set of atoms and let 
$\Aut(A)$ be the group of all permutations of $A$. Then
\[
\nG:=\blcb H\in\Aut(A):\forall n\in\omega\,(H\res{A_n}\in\nG_n)\brcb
\]
is a group of permutations of $A$. 
The sets in $\modcV_{\NrevD}$ are subsets of $\modcV_{\NrevD}$ with finite support.

\begin{prp} Let\/ $A$ be the set of
atoms of\/ $\modcV_{\text{\rm\NrevD}}$ and let\/ $\fm:=|A|$. Then
$$\modcV_{\text{\rm\NrevD}}\models \fm^2<\iseq(\fm)<[\fm]^2<\fin(\fm)\,.$$ 
\end{prp}

\begin{proof}
In~\cite{weird} it is shown that $\modcV_{\NrevD}\models\seq(\fm)<[\fm]^2$, 
which implies that $\modcV_{\NrevD}\models\iseq(\fm)<[\fm]^2$. Thus, since
$\fm^2\le \iseq(\fm)$ and $[\fm]^2\le\fin(\fm)$, it remains to show that
in $\modcV_{\NrevD}$ we have $\fm^2\neq\iseq(\fm)$ and $[\fm]^2\neq\fin(\fm)$.
\medskip

\noindent $\fm^2\neq\iseq(\fm)$: We show that there is no
injection $g_1:\iseq(A)\to A^2$. Assume towards a contradiction
that there is such an injection with finite support $E_1$.

By extending $E_1$ if necessary, we may assume that if 
\hbox{$(n+1,\langle a_0,\ldots,a_{l-1}\rangle,\varepsilon)\in E_1$}, 
then also $a_0,\ldots,a_{l-1}$ belong to $E_1$ as well as the
atom $(n+1,\langle a_0,\ldots,a_{l-1}\rangle,1-\varepsilon)$. We say that a finite subset of $A$ satisfying this condition is $\textit{closed}$.

For a large enough number $k\in\omega$ choose a $k$-element set
$X\subs A_0\setminus E_1$ such that
$|\iseq\left(X\right)|>\big{|}\left(E_1\cup X\right)^2\big{|}$.
Notice that $|\iseq\left(X\right)|\ge k!$ and 
that $\big{|}\left(E_1\cup X\right)^2\big{|}=\left(|E_1|+k\right)^2$.
Thus, we find a sequence $s\in\iseq(X)$ such that $g_1(s)\notin 
\left(E_1\cup X\right)^2$. So, there exists a $\pi\in\Fix_{\nG}(E_1\cup X)$ 
such that $\pi g_1(s)\neq g_1(s)$ but $\pi s=s$, which contradicts 
the fact that $E_1$ is a support of~$g_1$.\smallskip

\noindent $[\fm]^2\neq\fin(\fm)$: We show that there is no
injection $g_2:\fin(A)\to [A]^2$. Assume towards a contradiction
that there is such an injection with closed finite support~$E_2$.
For a large enough number $k\in\omega$ we choose again a $k$-element set
$X\subs A_0\setminus E_2$ such that
$|\fin\left(X\right)|>\big{|}\left[E_2\cup X\right]^2\big{|}$ and such that we can find a subset $S\subseteq X$ with $P:=g_2(S)\setminus(E_2\cup X)\neq\emptyset$ and $|S|\geq 2$. If $|P|=1$ it is clear that we can find a permutation $\pi\in G$ which fixes $E_2$ pointwise and for which $\pi(S)=S$ and $\pi(g_2(S))\neq g_2(S)$. Likewise, we find a contradiction also if $|P|=2$ and $P$ is not in the form $\{(l, p, 0),\,(l, p, 1)\}$ for some $l\in\omega\setminus\{0\}$ and $p\in \bigcup_{j=0}^{\,l}A_{l-1}^j$, so let us assume that $P$ is indeed in that form. Consider the extension $P'$ of $P$ to the smallest closed superset $P\subseteq P'$. If we can find $x\in S\setminus P'$ then $G$ contains a permutation $\pi$ with $\pi(x)\notin S$ and $\pi(P)=P$, which is a contradiction, so $S\setminus P'$ must be empty. We notice that if $a\in A$ is an atom, $Y$ is the smallest closed set of atoms containing $a$ and $b\in A_0\cap Y$, then for all permutations $\pi\in G$ we have that $\pi(b)\neq b$ implies $\pi(a)\notin Y$. We can now conclude since every element of $S$ is in the closure of $P$ and we can find a permutation $\pi\in G$ which fixes pointwise $E_2$ with $\pi(S)=S$ but for which $S$ is not fixed pointwise.
\end{proof}

So, the permutation model $\modcV_{\NrevD}$ witnesses the following

\begin{conc} 
The existence of an infinite cardinal\/ $\fm$ satisfying
%$$\fm^2<\iseq(\fm)<[\fm]^2<\fin(\fm)$$
\[\xymatrix@R=10mm@C=3mm{
\fin(\fm)\ar@{<-}[d]& &\iseq(\fm)\\ 
[\fm]^2\ar@{<-}[urr]& &\fm^2\ar@{->}[u]
}
\]
is consistent with\/ $\ZF$.
\end{conc}
              % fertig
\subsection{A Model for Diagram \protect\CrevD} 

We show that Diagram~\protect{\CrevD} holds in a permutation model 
$\modcV_{\CrevD}$ which is similar to the {\it Second Fraenkel Model}, 
where $\fm$ is the cardinality of the set of atoms of $\modcV_{\CrevD}$. 

The permutation model $\modcV_{\CrevD}$ is constructed as follows
(see also \cite[p.\,197]{cst}):
The set of atoms of the model $\modcV_{\CrevD}$ consists of
countably many mutually disjoint, cyclically ordered $3$-element sets. 
More formally,
\[
A=\bigcup_{n\in\omega} P_n,
\quad
\mbox{where }P_n=\{a_n,b_n,c_n\}\ (\mbox{for }n\in\omega),
\]
and the cyclic ordering on $P_n$ is illustrated by the following figure:
\begin{figure}[!ht]
\centering
\includegraphics[scale=.9]{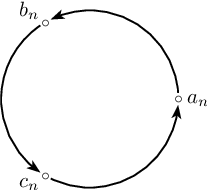}
\end{figure}

On each triple $P_n$, we define the cyclic distance between two elements
by stipulating $$\cy{a_n}{b_n}=\cy{b_n}{c_n}=\cy{c_n}{a_n}=1$$ and 
$$\cy{a_n}{c_n}=\cy{b_n}{a_n}=\cy{c_n}{b_n}=2\,.$$

Let $\nG$ be the group of those permutations of $A$ which preserve the
triples $P_n$ (\ie, $\pi P_n=P_n$ for $\pi\in\nG$ and $n\in\omega$)
and their cyclic ordering.
The sets in $\modcV_{\CrevD}$ are subsets of $\modcV_{\CrevD}$ with finite support.

\begin{prp} Let\/ $A$ be the set of
atoms of\/ $\modcV_{\text{\rm\CrevD}}$ and let\/ $\fm:=|A|$. Then
$$\modcV_{\text{\rm\CrevD}}\models [\fm]^2<\fm^2<\iseq(\fm)<\fin(\fm)\,.$$ 
\end{prp}

\begin{proof} We first show that $[\fm]^2\le\fm^2$,
and $\iseq(\fm)\le\fin(\fm)$, and then we show that
$[\fm]^2\neq\fm^2$, $\fm^2\neq\iseq(\fm)$, and $\iseq(\fm)\neq\fin(\fm)$.
\medskip

\noindent $[\fm]^2\le\fm^2$: We define an injective function $f_1:[A]^2\to A^2$.
Let $\{x,y\}\in [A]^2$ and $m,n\in\omega$ be 
such that $x\in P_m$ and $y\in P_n$. Without loss of generality we may assume
that $m\le n$. If $m<n$, then $f_1(\{x,y\}):=\la x,y\ra$, and if $m=n$, then
$f_1(\{x,y\}):=\la z,z\ra$ where $z:=P_m\setminus\{x,y\}$. It is easy to see 
that $f_1$ is an injective function, and since $f_1$ has empty support, $f_1$
belongs to~$\modcV_{\CrevD}$.\smallskip

\noindent $\iseq(\fm)\le\fin(\fm)$: We define an injective function $f_3:\iseq(A)\to\fin(A)$.
First, let $f_3(\la\;\ra):=\emptyset$. Now, let $s=\la x_0,\ldots,x_{k-1}\ra\in\iseq(A)$ be 
a non-empty sequence without repetition of length~$k$. Let $j:k\to\omega$ be such that 
for each $i\in k$, $x_i\in P_{j(i)}$. Let $E_0:=\emptyset$, and by induction, for $i\in k$ define
$$E_{i+1}:=
\begin{cases} 
E_i\cup\{x_{i}\} &\text{if $P_{j(i)}\cap E_i=\emptyset$,}\\[1.2ex]
E_i &\text{otherwise,}
\end{cases}
$$
and
$$\epsilon_{i}:=
\begin{cases} 
2 &\text{if $P_{j(i)}\cap E_i=\{u\}$ and $\cy{u}{x_i}=2$,}\\[1.2ex]
1 &\text{otherwise.}
\end{cases}
$$
Furthermore, let $\sigma(0):=j(0)$, and by induction, for $i\in k-1$ define
$\sigma(i+1):=\sigma(i)+j(i+1)+1$. Finally, let $\{p_i:i\in\omega\}$ be
an enumeration of the prime numbers and let $$q_s:=\prod_{i\in k}p_{\sigma(i)}^{\epsilon_i}\,.$$
Now, we define $f_3(s):=E_{k}\cup P_{q_s}$. 
It is easy to see that $f_3$ is an injective function from 
$\iseq(A)$ into $\fin(A)$, and since $f_3$ has empty support, $f_3$
belongs to~$\modcV_{\CrevD}$.\medskip

\noindent $[\fm]^2\neq\fm^2$: It is enough to show that there is no
injection from $A^2$ into $[A]^2$. Assume towards a contradiction
that there exists an injection $g_1:A^2\to [A]^2$ with finite support $E_1$.

Let $E_1\subs E$ for some non-empty set $E\in\fin(A)$ such that $P_n\subs E$ whenever $E\cap P_n\neq\emptyset$. Then $E$ is also a 
support of~$g_1$. Now $|[E]^2|<|E^2|$, which implies that there exists
a pair $\la x,y\ra\in E^2$ such that $g_1(\la x,y\ra)\notin [E]^2$.
So, there exists a $\pi\in\Fix_{\nG}(E)$ such that $\pi g_1(\la x,y\ra)\neq
g_1(\la x,y\ra)$, but $\pi\la x,y\ra=\la x,y\ra$, which contradicts the fact that
$E$ is a support of $g_1$.\smallskip

\noindent $\fm^2\neq\iseq(\fm)$: It is enough to show that there is no
injection from $\iseq(A)$ into $A^2$. Assume towards a contradiction
that there exists an injection $g_2:\iseq(A)\to A^2$ with finite support $E_2$.
Let $E_2\subs E$ for some non-empty set $E\in\fin(A)$ such that $P_n\subs E$ whenever $E\cap P_n\neq\emptyset$. Then $E$ is also a 
support of~$g_2$. Now $|E^2|<|\iseq(E)|$, and by similar arguments as above,
we obtain a contradiction.

\noindent $\iseq(\fm)\neq\fin(\fm)$: It is enough to show that there is no
injection from $\fin(A)$ into $\iseq(A)$. Assume towards a contradiction
that there is an injection \hbox{$g_3:\fin(A)\to\iseq(A)$} with finite support $E_3$, where we can assume that for all $n\in\omega$, $P_n\subs E_3$ whenever $E_3\cap P_n\neq\emptyset$.

Since $E_3$ is finite, there exists an $n\in\omega$ such that $g_3(P_n)\notin\iseq(E_3)$.
Let $a\in A$ be the first element of the sequence $g_3(P_n)$ which does not belong to~$E_3$. 
Then we find a $\pi\in\Fix_{\nG}(E_3)$ such that $\pi a\neq a$ (which implies $\pi g_3(P_n)\neq
g_3(P_n)$) but $\pi P_n=P_n$, which contradicts the fact that $E_3$ is a support of~$g_3$.
\end{proof}

So, the model $\modcV_{\CrevD}$ witnesses the following

\begin{conc} 
The existence of an infinite cardinal\/ $\fm$ satisfying
%$$[\fm]^2<\fm^2<\iseq(\fm)<\fin(\fm)$$
\[\xymatrix@R=10mm@C=3mm{
\fin(\fm)\ar@{<-}[rr]& &\iseq(\fm)\\ 
[\fm]^2& &\fm^2\ar@{<-}[ll]\ar@{->}[u]
}
\]
is consistent with\/ $\ZF$.
\end{conc}
              % fertig
\subsection{A Model for Diagram \protect\ZrevD} 

As mention above, Diagram~\protect{\ZrevD}
holds in the {\it Ordered Mostowski Model\/}, 
where $\fm$ is the cardinality of the set of atoms (see, for example, 
\cite[Related\;Result\,48,\;p.\,217]{cst}). This
leads to the following

\begin{conc} 
The existence of an infinite cardinal\/ $\fm$ satisfying
%$$[\fm]^2<\fm^2<\fin(\fm)<\iseq(\fm)$$
\[\xymatrix@R=10mm@C=3mm{
\fin(\fm)\ar@{->}[rr]& &\iseq(\fm)\\ 
[\fm]^2& &\fm^2\ar@{<-}[ll]\ar@{->}[ull]
}
\]
is consistent with\/ $\ZF$.
\end{conc}              % fertig
\section{On Diagram $\CDiag$} % C

Similar as in the proof of {\LEM}\;\ref{lem:N}, in every model for
Diagram\;$\CDiag$, 
we have that the cardinality $\fm$ is transfinite.

\begin{lem}\label{lem:C}
If\/ $\fm^2\le[\fm]^2$ and\/ $\fin(\fm)\le\iseq(\fm)$ for some\/ $\fm\ge\fone$, 
then\/ $\aleph_0\le\fm$.
\end{lem}

\begin{proof} Assume that $\fm^2\le[\fm]^2$ and $\fin(\fm)\le\iseq(\fm)$ 
for some cardinal $\fm\ge\fone$
and let $A$ be a necessarily infinite set with $\card{A}=\fm$. 
Let $f:A^2\to [A]^2$ and $g:\fin(A)\to\iseq(A)$ be injections.
The goal is to construct with the functions $f$ and $g$ an injection $h:\omega\to A$.
We first construct a countably infinite set of pairwise disjoint 
non-empty finite subsets of $A$. For this, we first
choose an element $a_0\in A$, let $E_0:=\{a_0\}$, and let $\nE_0:=\{E_0\}$.

Assume that for some $n\in\omega$ 
we have already constructed an $(n+1)$-element set $\nE_n:=\blcb E_i:i\le n\brcb$
of pairwise disjoint non-empty finite subsets of $A$. 
Let $$E_{n+1}:=\bigcup_{i,j\le n}\!\Big{\{}x:\exists a\in E_i\,\exists b\in E_j\,\exists y
\bigl(f(\la a,b\ra)=\{x,y\}\bigr)\Big{\}}\setminus\bigcup_{i\le n}E_i\,,$$
and let $\nE_{n+1}:=\nE_n\cup \{E_{n+1}\}$. Notice that for $k:=|\bigcup_{i\le n}E_i|$,
we have $$\bigg{|}\biggl{(}\bigcup_{i\le n}E_i\biggr{)}{\,}^2\bigg{|}\;=\;k^2\,>\,\binom k2
\;=\;\bigg{|}\biggl[\,\bigcup_{i\le n}E_i\biggr]^2\bigg{|}\,,$$ 
which implies that $E_{n+1}\neq\emptyset$. 
Proceeding this way,  $\blcb E_n:n\in\omega\brcb$ is a countably infinite 
set of pairwise disjoint non-empty finite subsets of $A$.

Now, we apply the function $g$. For every $n\in\omega$, let $S_n:=g(E_n)$. 
Furthermore, let $\nS_0:=S_0$, and in general, for $n\in\omega$ let $\nS_{n+1}:=
\nS_n{}^\frown S_{n+1}$. In this way, we obtain an infinite sequence
$\nS_\infty$ of elements of~$A$. Since $g$ is injective and the non-empty finite sets
$E_n$ are pairwise disjoint, the sequence $\nS_\infty$ must contain infinitely
many pairwise distinct elements of~$A$. Now, let $h$ be the enumeration of these
pairwise distinct elements in the order they appear in $\nS_\infty$. Then
$h:\omega\to A$ is an injection.
\end{proof}

As a consequence of {\PRP}\;\ref{prp:fin-to-one} and {\LEM}\;\ref{lem:C} we get

\begin{cor}
If\/ $\fm^2\le[\fm]^2$ and\/ $\fin(\fm)\le\iseq(\fm)$ 
for some\/ $\fm=\card A\ge\fone$, then there 
exists a finite-to-one function\/ $g:\seq(A)\to\fin(A)$.
\end{cor}
\vspace{-1ex}
{\bf Acknowledgement:} 
We would like to thank the referee for her or his careful reading
and the numerous comments and corrections that helped to improve 
the quality of this article.
\vspace{-3ex}

%%%%%%%%%%%%%%%%%%%%%%%%%%

%%%%%%%%%%%%%%%%%%%%%%%%%%%%%%%%%%%%%%

\small{Department of Mathematics\\
ETH Zentrum\\
R\"amistrasse\;101\\
8092 Z\"urich, Switzerland}\bigskip

\small{Einstein Institute of Mathematics\\
The Hebrew University of Jerusalem\\
Givat Ram\\
9190401 Jerusalem, Israel\\[1.2ex]
{\sl and}\\[1.2ex]
Department of Mathematics\\
Hill Center, Busch Campus\\
110 Frelinghuysen Road\\
Piscataway, NJ 08854-8019, U.S.A.
}


\begin{thebibliography}{8}

\bibitem{carla}
{\sc Lorenz Halbeisen and Saharon Shelah}, 
{\it Consequences of arithmetic for set theory\/ {\rm [Sh:488]}},
{\bf The Journal of Symbolic Logic},
59\;(1994), 30--40.

\bibitem{berlin}
{\sc Lorenz Halbeisen and Saharon Shelah}, 
{\it Relations between some cardinals in the absence of
               the axiom of choice\/ {\rm [Sh:699]}}, 
{\bf The Bulletin of Symbolic Logic},
7\;(2001), 237--261.

\bibitem{weird}
{\sc Lorenz Halbeisen}, 
{\it A weird relation between two cardinals}, 
{\bf Archive for Mathematical Logic},
57\;(2018), 593--599.

\bibitem{cst1st}
{\sc Lorenz Halbeisen}, 
{\bf Combinatorial Set Theory: with a gentle introduction to forcing}, 
{\rm (1st.\,ed.)},
{\rm [Springer Monographs in Mathematics]},
{\rm Springer-Verlag},
{\rm London}~(2012).

\bibitem{cst}
{\sc Lorenz Halbeisen}, 
{\bf Combinatorial Set Theory: with a gentle introduction to forcing}, 
{\rm (2nd.\,ed.)},
{\rm [Springer Monographs in Mathematics]},
{\rm Springer-Verlag},
{\rm London}~(2017).

%\bibitem{Hodges}
%{\sc Wilfrid Hodges}, 
%{\bf Model Theory}, 
%{\rm [Encyclopedia of Mathematics and its Applications]},
%{\rm Cambridge University Press},
%{\rm Cambridge}~(1993).

\bibitem{Jechchoice}
{\sc Thomas Jech}, 
{\bf The Axiom of Choice}, 
{\rm [Studies in Logic and the Foundations of Mathematics\;75]},
{\rm North-Holland},
{\rm Amsterdam}~(1973).

%\bibitem{Shen1}
%{\sc Guozhen Shen}, 
%{\it Generalizations of Cantor's theorem in $\ZF$}, 
%{\bf Mathematical Logic Quarterly},
%63\;(2017), 428--436.

\bibitem{Shen1}
{\sc Guozhen Shen}, 
{\it Factorials of infinite cardinals in\/ {\rm ZF}. Part\;I: {$\textsf{ZF}$} results}, 
{\bf The Journal of Symbolic Logic},
85\;(2020), 224--243.

\bibitem{Shen2}
{\sc Guozhen Shen}, 
{\it Factorials of infinite cardinals in\/ {\rm ZF}. Part\,II: Consistency results}, 
{\bf The Journal of Symbolic Logic},
85\;(2020), 244--270.

\end{thebibliography}
\end{document}